\documentclass[reqno]{amsart}
\usepackage[letterpaper,top=2cm,bottom=2cm,left=3cm,right=3cm,marginparwidth=2.0cm]{geometry}

\usepackage{amsmath,amsthm,amssymb,amscd,mathtools,mathrsfs,booktabs}

\usepackage{tikz-cd} 
\usepackage{graphicx, psfrag, array}
\usepackage{lipsum}
\usepackage{subcaption} 
\usepackage[bookmarksnumbered, colorlinks, plainpages]{hyperref}
\hypersetup{colorlinks=true,linkcolor=red, anchorcolor=green, citecolor=cyan, urlcolor=red, filecolor=magenta, pdftoolbar=true}
\usepackage{shuffle} 
\usepackage{titletoc}
\usepackage{ytableau}

\newcommand{\defin}[1]{%
\relax\ifmmode%
\textcolor{blue}{#1}%
\else\textcolor{blue}{\emph{#1}}%
\fi%
}

\usepackage[colorinlistoftodos]{todonotes}
\definecolor{melon}{rgb}{0.99,0.72,0.67}
\definecolor{cornflower}{rgb}{0.66,0.77,0.95}
\definecolor{lightgreen}{rgb}{0.67, 0.99, 0.67}
\definecolor{aqua}{rgb}{0, 1, 1}

\newcommand{\DES}{{\mathsf{DES}}}
\newcommand{\des}{{\mathsf{des}}}

\newcommand{\inv}{{\mathsf{inv}}}
\newcommand{\exc}{{\mathsf{exc}}}

\newcommand{\maj}{{\mathsf{maj}}}

\newcommand{\DEX}{{\mathsf{DEX}}}

\newcommand{\ch}{{\mathrm{ch}}}
\newcommand{\ps}{{\mathrm{ps}_q}}
\newcommand{\psm}{{\mathrm{ps}_{q,m}}}

\newcommand{\rk}{{\mathrm{rk}}}

\newcommand{\Stab}{{\mathrm{Stab}}}

\newcommand{\cro}{{\mathsf{cro}}}
\renewcommand{\Dot}{{\mathsf{Dot}}}

\renewcommand{\S}{{\mathfrak{S}}}

\newcommand{\x}{{\mathbf{x}}}
\renewcommand{\L}{{\mathcal{L}}}

\newcommand{\Hilb}{{\mathsf{Hilb}}}
\newcommand{\grFrob}{{\mathsf{grFrob}}}
\newcommand{\qbin}[2]{{#1 \brack #2}_q} 
\newcommand{\floor}[1]{\lfloor #1 \rfloor}
\newcommand{\FY}{\operatorname{FY}}

\newtheorem{thm}{Theorem}[section]
\newtheorem{cor}[thm]{Corollary}
\newtheorem{conj}[thm]{Conjecture}
\newtheorem{prop}[thm]{Proposition}
\newtheorem{lem}[thm]{Lemma}

\newtheorem{prob}[thm]{Problem}
\theoremstyle{definition}
\newtheorem{defi}[thm]{Definition}

\newtheorem{exa}[thm]{Example}
\newtheorem{notation*}[thm]{Notation}

\theoremstyle{remark}
\newtheorem{rem}{Remark}[section]

\title[Equivariant gamma-positivity]{Equivariant gamma-positivity of matroid Chow rings}

\author{Hsin-Chieh Liao}
\address{Department of Mathematics, Washington University in St. Louis, One Brookings Drive, St. Louis, Missouri 63130, USA}
\email{liaoh@wustl.edu}

\date{}

\begin{document}

\maketitle

\begin{abstract}
In this paper, we prove that the Chow ring and augmented Chow ring of a matroid are equivariantly $\gamma$-positive under the action of any group of automorphisms. 
Our approach provides an explicit combinatorial interpretation of the coefficients in the equivariant $\gamma$-expansion, which is new even in the non-equivariant setting.
This result confirms a conjecture of Angarone, Nathanson, and Reiner, and extends the author’s previous work on the positivity of equivariant Charney--Davis quantities for matroids. 
Specializing our formulas to uniform matroids, we obtain representation-theoretic interpretations that extend the Schur-$\gamma$-positivity results of Shareshian and Wachs for Eulerian and binomial Eulerian quasisymmetric functions. 
Finally, we address a problem posed by Athanasiadis by giving a combinatorial interpretation of a $(p,q)$-analog of the $\gamma$-expansion of the binomial Eulerian polynomial.

\medskip

\noindent\textbf{AMS Classification 2020: 05B35, 05E18, 05E14, 05E05}

\end{abstract}

\dottedcontents{section}[0em]{}{0em}{0.5pc}
\dottedcontents{subsection}[0.5em]{}{0em}{0.5pc}

\setcounter{tocdepth}{1} 
\tableofcontents

\section*{Notation}
We fix some notations that we will use later. For any positive integer $n$ and integers $1\le a,b\le n$,
\begin{itemize}
    \item $[n]\coloneqq\{1,2,\ldots,n\}$ and $2^{[n]}$ is the collection of all subsets of $[n]$.
    \item $[a,b]\coloneqq \{a,a+1,\ldots,b-1,b\}$ if $a\le b$, and $[a,b]\coloneqq\emptyset$ if $a>b$.
    \item $[n]_t\coloneqq 1+t+\ldots+t^{n-1}$ for $n\ge 1$ and $[0]_t\coloneqq 0$.   
\end{itemize}
For a subset $S\subseteq [n]$, let $\Stab(S)$ denote the collection of subsets of $S$ containing no consecutive integers.

\section{Introduction}
Given a finite sequence, there are many features that one can ask whether the sequence possesses. A sequence $\{a_i\}_{i=0}^d$ or a polynomial $f(t)=\sum_{i=0}^d a_i t^i$ of degree $d$ is said to be \defin{palindromic} if 
\[
    a_j=a_{d-j}\quad \text{ for all } j=0,1,\ldots, d.
\]
It is said to be \defin{unimodal} if there is some $0\le j\le d$ such that 
\[
    a_0\le a_1\le\cdots\le a_j\ge\cdots\ge a_{d-1}\ge a_d.
\]
A palindromic polynomial $f(t)$ can be expressed uniquely as 
\[
    f(t)=\sum_{k=0}^{\lfloor\frac{d}{2}\rfloor}\gamma_{k}t^k(1+t)^{d-2k}.
\]
The above expression is known as the $\gamma$-expansion of $f$ and $(\gamma_1,\gamma_2,\ldots, \gamma_{\lfloor\frac{d}{2}\rfloor})$ is called the $\gamma$-vector of $f$.
We say the polynomial $f$ is \defin{$\gamma$-positive} (or \defin{$\gamma$-nonnegative}) if $\gamma_k\ge 0$ for all $k$. 
It is not hard to see that a polynomial being $\gamma$-positive implies that it is both palindromic and unimodal. 

There are several reasons why the $\gamma$-positivity of a polynomial has attracted significant attention. One is the abundance of well-studied combinatorial sequences that are $\gamma$-positive; in many cases, their $\gamma$-coefficients admit interesting combinatorial interpretations, and sometimes this positivity phenomenon together with these interpretations can be lifted to a richer algebraic or representation-theoretic setting. A classical example is the $\gamma$-positivity of the Eulerian polynomial, first studied by Foata and Strehl \cite{FoataStrehl1974}, which was later extended to the level of symmetric functions \cite{ShareshianWachs2010,ShareshianWachs2020}.
On the other hand, $\gamma$-positivity also appears naturally in geometric contexts. In particular, Gal conjectured that the $h$-polynomial of every flag simplicial sphere is $\gamma$-positive \cite[Conjecture 2.1.7]{Gal2005}.
In this framework, the Eulerian polynomial arises as the $h$-polynomial of the dual permutahedron, which is itself a flag simplicial sphere.
For more details about $\gamma$-positivity, we refer the readers to the survey \cite{AthanasiadisGamma} by Athanasiadis.

Let $M$ be a loopless matroid on the ground set $E$. The \defin{Chow ring of the matroid $M$} encodes the structure of the lattice of flats $\L(M)$ and is defined by
\[
    A(M)\coloneqq \mathbb{R}[x_{F}: F\in\L(M)\setminus\{\emptyset\}]/(I+J)
\]
where $I=(x_Fx_{F'}:F,F' \text{ not comparable})$ and $J=(\sum_{F:i\in x_F}x_F: i\in E)$. Similarly, the \defin{augmented Chow ring of $M$} is defined by
\[
    \widetilde{A}(M)\coloneqq \mathbb{R}[\{x_F:F\in\L(M)\setminus\{[n]\}\}\cup\{y_i:i\in E\}]/(\widetilde{I}+\widetilde{J})
\]   
where $\widetilde{I}=(x_Fx_{F'}:F,F' \text{ not comparable})+(y_ix_F: i\notin F)$ and $\widetilde{J}=(y_i-\sum_{F:i\notin F}x_F: i\in E)$. 
The ordinary Chow ring $A(M)$ can be recovered from the augmented one by quotienting out the ideal generated by the variables $y_i$.
Despite the seemingly complicated presentation, the Chow ring $A(M)$ has a nice $\mathbb{R}$-basis, which we refered to as the \defin{Feichtner-Yuzvinsky basis}
\[
    \FY(M)\coloneqq\left\{x_{F_1}^{a_1}x_{F_2}^{a_2}\ldots x_{F_\ell}^{a_\ell}:\substack{~\emptyset=F_0\subsetneq F_1\subsetneq F_2\subsetneq\ldots\subsetneq F_\ell \text{ for } 0\le \ell\le n \\
    1\le a_i\le \rk_M(F_i)-\rk_M(F_{i-1})-1}\right\}.
\]
It is the standard monomials obtained from the Gr\"{o}bner basis of the ideal $I+J$ found by Feichtner and Yuzvinsky \cite{FeichtnerYuzvinsky2004}. A similar Feichtner-Yuzvinsky basis for $\widetilde{A}(M)$, 
\[
	\widetilde{\FY}(M)\coloneqq\left\{x_{F_1}^{a_1}x_{F_2}^{a_2}\ldots x_{F_\ell}^{a_\ell}: \substack{\emptyset\subsetneq F_1\subsetneq F_2\subsetneq\ldots\subsetneq F_\ell \text{ for }0\le\ell\le n\\
    1\le a_1\le\rk_M(F_1),\\ 
    1\le a_i\le\rk_M(F_i)-\rk_M(F_{i-1})-1 \text{ for }i\ge 2}
    \right\},
\]
was found by the author \cite{Liao2023FPSAC,Liao2023One} and independently by Chris Eur (see \cite{Mastroeni2022Koszul}).
Using these bases, one can compute the Hilbert polynomials of the corresponding Chow rings and augmented Chow rings, which are now commonly referred to as the \defin{Chow polynomial} and the \defin{augmented Chow polynomials of a matroid}.

The Chow rings and augmented Chow rings of matroids play central roles in the resolution of the long-standing Rota-Welsh conjecture and Dowling-Wilson top-heavy conjecture, respectively. A key ingredient in both proofs is that these rings satisfy the so-called K\"{a}hler package (see \cite{AHK2018} for Chow rings and \cite{BHMPW2020+} for augmented Chow rings). 
This package includes a Poincar\'{e} duality, which implies the Chow and augmented Chow polynomials are palindromic, and the Hard Lefschetz Theorem, which implies these polynomials are unimodal. 
The Chow and augmented Chow polynomials of a matroid were proved to be $\gamma$-positive by Ferroni, Matherne, Stevens, and Vecchi \cite{Ferroni2024hilbert}, and independently by Wang \cite[p.33]{Ferroni2024hilbert}, using a recursive argument based on the semismall decomposition introduced in \cite{BHMPW2020}. 
However, an explicit expression or a combinatorial interpretation of the $\gamma_k$ coefficients remains unknown.
Moreover, these polynomials were conjectured to satisfy an even stronger positivity phenomenon. 
This was formulated by Huh and Stevens \cite[Conjecture 4.3.3]{StevenThesis}, and independently by Ferroni and Schr\"{o}ter \cite[Conjecture 8.18]{FerroniSchroter2024} (for Chow polynomials) as follows. 
\begin{conj}\label{conj: Real-rootedChow}
For every loopless matroid $M$, the Chow polynomial and the augmented Chow polynomial of $M$ have only real roots.  
\end{conj}

Taking into account group actions on matroids, Angarone, Nathanson, and Reiner \cite{ANR2023PermRepchow} extended several known and conjectural inequalities for Chow polynomials to the equivariant (representation-theoretic) and Burnside (permutation-representation) setting. 
In particular, they conjectured that the Chow ring $A(M)$, viewed as a graded $\mathbb{C}G$-modules, is equivariantly $\gamma$-positive (see Section \ref{sec:EqGamma} for the definition), and moreover equivariantly $PF_{\infty}$ \cite[Conjecture 4.4 (ii)]{ANR2023PermRepchow}, which is an equivariant analog of Conjecture \ref{conj: Real-rootedChow}.

This paper is the full version of the extended abstract \cite{Liao2024EqGamma} presented at FPSAC 2025.
We prove the equivariant $\gamma$-positivity conjecture of Angarone, Nathanson, and Reiner. Our approach further provides an explicit interpretation of the $\gamma$-coefficients, which is new even in the non-equivariant setting. We expect that this interpretation will provide further insight into Conjecture \ref{conj: Real-rootedChow} and its equivariant analog.

In Section \ref{sec:EqGamma}, we establish the equivariant $\gamma$-positivity conjecture. In Section \ref{sec:Uniform}, we apply our $\gamma$-positivity formulas to uniform matroids. 
From the perspective of matroid Chow rings, our results extend the combinatorial interpretations of the Charney-Davis quantity for uniform matroids studied by Hameister, Rao, and Simposon \cite{HameisterRaoSimpson2021}, as well as for general matroids by Liao \cite{Liao2023Two}. 
From the perspective of symmetric functions and representations of symmetric groups, our results generalize the Schur-$\gamma$-positivity of the Eulerian and binomial Eulerian quasisymmetric functions studied by Shareshian and Wachs \cite{ShareshianWachs2010, ShareshianWachs2020}, Linusson, Shareshian, and Wachs \cite{LinussonShareshianWachs2012}, and Athanasiadis \cite{AthanasiadisGamma}. 
In Section \ref{sec:AthanasiadisProblem}, we apply results from Section \ref{sec:Uniform} to address a problem raised by Athanasiadis \cite[Problem 2.5]{AthanasiadisGamma} concerning a $(p,q)$-analog of the $\gamma$-expansion of the binomial Eulerian polynomial. 
In Section \ref{sec: Braid}, we comment on applications of our equivariant $\gamma$-positivity results on braid matroids.
In Section \ref{sec: Developments}, we discuss recent developments related to this work that have appeared since the preliminary version was posted on \texttt{arXiv}.

In the course of preparing this paper, we learned that Stump \cite{StumpChow2024} independently obtained an interpretation of the (non-equivariant) $\gamma$-coefficients for Chow and augmented Chow polynomials of matroids. His approach uses a specialization of the \emph{Poincar\'{e} extended $\mathbf{ab}$-index} introduced in \cite{DorpalenMaglioneStump2023ab}, expressing the $\gamma$-coefficient as the number of maximal chains in $\L(M)$ with prescribed descent set under an $R$-labeling. His results agree with ours, as a consequence of Stanley \cite[Theorem 3.1]{StanleyRlabel1974}. 
In future work, we aim to further investigate the relationship between the Feichtner-Yuzvinsky basis and the Poincar\'{e} extended $\mathbf{ab}$-index and study the equivariant analog of the extended $\mathbf{ab}$-index.

\section{Group actions on posets}
Let $P$ be a finite poset with a unique minimal element $\hat{0}$ and a unique maximal element $\hat{1}$. If $P$ is graded of rank $n$ with rank function $\rk_P: P\longrightarrow\{0,1,\ldots,n\}$, then for $S\subseteq [n-1]$, the \defin{rank selected subposet} of $P$ is
\[
    P_S\coloneqq\{x\in P~:~\rk_P(x)\in S\}\cup\{\hat{0},\hat{1}\}.
\]
Consider a group $G$ of automorphisms of $P$ acting on the poset $P$. The action of $G$ preserves the rank of the elements in $P$; hence for any $S\subseteq [n-1]$, the group $G$ permutes the maximal chains in $P_S$. Let us denote $\alpha_P(S)$ the permutation representation of $G$ generated by the maximal chains in $P_S$. Consider the virtual representation
\[
    \beta_P(S)\coloneqq\sum_{T\subseteq S}(-1)^{|S|-|T|}\alpha_P(T),
\]
which by M\"{o}bius inversion, is equivalent to
\begin{equation} \label{eq: flagFtoH}
    \alpha_P(S)=\sum_{T\subseteq S}\beta_P(T).
\end{equation}
Here $\alpha_P(S)$ and $\beta_P(S)$ can be viewed as the equivariant analogs of the flag $f$-vector and flag $h$-vector of the poset $P$, respectively. See \cite[Section 3.13]{StanleyEC1} for more details.

\begin{thm}[Stanley \cite{Stanley1982GroupPoset}] \label{thm:BetaGenRep}
If $P$ is Cohen-Macaulay, then $\beta_P(S)$ is a genuine representation of $G$ and 
\[
    \beta_P(S)\cong_G \tilde{H}_{|S|-1}(P_S).
\]
\end{thm}
Throughout this paper, the poset $P$ we care about is always the lattice of flats of a matroid. It is well known that the lattice of flats of a matroid is always Cohen-Macaulay.

\section{Equivariant $\gamma$-positivity}\label{sec:EqGamma}

Let $G$ be a finite group. For a $\mathbb{C}G$-module $M$, let $[M]$ be the isomorphism class of $\mathbb{C}G$-modules containing $M$. 

\begin{defi}
The Grothendieck ring $R_{\mathbb{C}}(G)$ of $\mathbb{C}G$-modules is a free abelian group having the transversal of isomorphism classes of simple $\mathbb{C}G$-modules $\{[S_i],\ldots,[S_{cc(G)}]\}$ as a $\mathbb{Z}$-basis, where $cc(G)$ is the number of conjugacy classes of $G$, with the addition and multiplication relations given by 
\[
    [S_i]+[S_j]\coloneqq [S_i\oplus S_j] \quad  \text{ and } \quad  [S_i]\cdot[S_j]\coloneqq[S_i\otimes_{\mathbb{C}} S_j]
\]
for $1\le i,j\le cc(G)$ and extended linearly over $\mathbb{C}$. Now every element $A$ in $R_{\mathbb{C}}(G)$ has a unique expression as $a=\sum_{i=1}^{cc(G)}a_i[S_i]$. We say $A$ is a \emph{genuine representation} of $G$ if $a_i\ge 0$ for all $i$, denoted by $A\ge_{R_{\mathbb{C}}(G)}0$.
\end{defi}

Given a graded $\mathbb{C}G$-module $V=\bigoplus_{i} V_i$, define the \defin{equivariant Hilbert series} of $V$ to be the formal power series 
\[
    \Hilb_{G}(V,t)=\sum_{i}[V_i] t^i\in R_{\mathbb{C}}(G)[[t]].
\]
 
\begin{defi}
Let $V=\bigoplus_{i=0}^d V_i$ be a finite dimensional graded $\mathbb{C}G$-module. 
We say that $V$ is \defin{($G$-)equivariant $\gamma$-positive} if its equivariant Hilbert series can be expressed as 
\[
    \Hilb_G(V,t)=\sum_{i=0}^d[V_i]t^i=\sum_{k=0}^{\lfloor\frac{d}{2}\rfloor}\gamma_k t^k(1+t)^{d-2k}
\]
and the uniquely defined coefficient $\gamma_k\in R_{\mathbb{C}}(G)$ is a class of a genuine representation of $G$ over $\mathbb{C}$ for all $k$, i.e. $\gamma_k\ge_{R_{\mathbb{C}}(G)}0$ for all $k$.    
\end{defi}

Let $M$ be a loopless matroid of rank $r$ on the ground set $E$ with the lattice of flats $\L(M)$. Write $A(M)_\mathbb{C}\coloneqq A(M)\otimes_{\mathbb{R}}\mathbb{C}$ and $\widetilde{A}(M)_\mathbb{C}\coloneqq \widetilde{A}(M)\otimes_{\mathbb{R}}\mathbb{C}$. 

From now on, let $G$ be a group of automorphisms of the matroid $M$.
The induced action of $G$ on the Chow ring $A(M)_{\mathbb{C}}=\bigoplus_{i=0}^{r-1}A_{\mathbb{C}}^i$ endows each graded piece $A_{\mathbb{C}}^i$ a $\mathbb{C}G$-module structure. The same holds for the augmented Chow ring $\widetilde{A}(M)_{\mathbb{C}}=\bigoplus_{i=0}^r\widetilde{A}^i_{\mathbb{C}}$. 
The following conjecture was proposed by Angarone, Nathanson, and Reiner.

\begin{conj}[{\cite[Conjecture 5.2] {ANR2023PermRepchow}}] \label{conj:Equiv}
For any (loopless) matroid $M$ of rank $r$ and any group $G$ of automorphisms of $M$, the matroid Chow ring $A(M)_{\mathbb{C}}=\bigoplus_{i=0}^{r-1}A^i_{\mathbb{C}}$ is ($G$-)equivariant $\gamma$-positive. 
\end{conj}

In the remainder of this section, we prove Conjecture \ref{conj:Equiv} in Theorem \ref{thm:GammaChow}, as well as its augmented counterpart in Theorem \ref{thm:GammaAug}.
Note that the Feichtner-Yuzvinsky bases $\FY(M)$ and $\widetilde{\FY}(M)$ form permutation bases of $A(M)_{\mathbb{C}}$ and $\widetilde{A}(M)_{\mathbb{C}}$, respectively, under the action of $G$.

\begin{prop} \label{prop:HilbChow}
Let $M$ be a loopless matroid of rank $r$.
\begin{itemize} 
    \item[(i)] The $G$-equivariant Hilbert series of $A(M)_{\mathbb{C}}$ is given by
    \[
        \Hilb_G(A(M)_{\mathbb{C}},t)=\sum_{S\subseteq [r-1]}\phi_{S,r}(t)[\alpha_{\L(M)}(S)],
    \]
    where for a subset $S=\{s_1<s_2<\cdots<s_{\ell}\}\subseteq [r-1]$,
    \[
        \phi_{S,r}(t)\coloneqq t^{|S|}[s_1-1]_t[s_2-s_1-1]_t\cdots[s_{\ell}-s_{\ell-1}-1]_t[r-s_{\ell}]_t.
    \]
    In particular, $\phi_{\emptyset,r}(t)=[r]_t$.
    \item[(ii)] The $G$-equivariant Hilbert series of $\widetilde{A}(M)_{\mathbb{C}}$ is given by
    \[
    \Hilb_G(\widetilde{A}(M)_{\mathbb{C}},t)=\sum_{S\subseteq[r-1]}\psi_{S,r}(t)[\alpha_{\L(M)}(S)],
    \]
    where for a subset $S=\{s_1<s_2<\cdots<s_{\ell}\}\subseteq [r-1]$,
    \[
    \psi_{S,r}(t)=\begin{cases}
        t^{|S|}[s_1]_t[s_2-s_1-1]_t\ldots[s_{\ell}-s_{\ell-1}-1]_t[r-s_{\ell}]_t & \text{ if }S\neq \emptyset\\
        [r+1]_t & \text{ if } S=\emptyset
    \end{cases}.
    \]
\end{itemize}

\end{prop}
\begin{proof}
For (i), recall that the Feichtenr-Yuzvinsky basis $\FY(M)$ consists of monomials $x_{F_1}^{a_1}x_{F_2}^{a_2}\ldots x_{F_{\ell}}^{a_{\ell}}$ indexed by chains (including the empty chain)
\[
    \emptyset\neq F_1\subsetneq F_2\subsetneq\cdots\subsetneq F_{\ell}\subseteq{E}
\]
such that the exponent $a_j$ satisfies $1\le a_j\le \rk(F_j)-\rk(F_{j-1})-1$ for $j=1,2,\ldots,\ell$, where we set $F_0=\emptyset$.
For any subset $S=\{s_1<s_2<\cdots<s_{\ell}\}\subseteq [r-1]$, the permutation $\mathbb{C}G$-module generated by chains in $\L(M)$ with rank set $S$ is
\begin{align*}
    \alpha_{\L(M)}(S)
    &=\mathbb{C}G\{F_1\subsetneq \cdots\subsetneq F_{\ell}: \rk(F_i)=s_i~\forall i\}.\\
    &\cong \mathbb{C}G\{F_1\subsetneq \cdots\subsetneq F_{\ell}\subsetneq E: \rk(F_i)=s_i~\forall i\}.
\end{align*}
Hence, as a graded $\mathbb{C}G$-module, $A(M)_{\mathbb{C}}$ decomposes as a direct sum of $\alpha_{\L(M)}(S)$ over all $S\subseteq [r-1]$.

Assume that $
    \Hilb_G(A(M)_{\mathbb{C}},t)=\sum_{S\subseteq [r-1]}\phi_{S,r}(t)[\alpha_{\L(M)}(S)]$
for some polynomials $\phi_{S,r}(t)$. To determine $\phi_{S,t}(t)$, we consider a map $f:\FY(M)\longrightarrow 2^{[r-1]}$ defined by 
\[
    f(x_{F_1}^{a_1}x_{F_2}^{a_2}\ldots x_{F_{\ell}}^{a_{\ell}}x_E^i)=\{\rk(F_1),\rk(F_2),\ldots,\rk(F_{\ell})\}
\]
where all flats $F_i$ are distinct and are not $E$. Since for any $S=\{s_1<\cdots<s_{\ell}\}\subseteq [r-1]$, the inverse image of $f$ on $S$ consists of monomials
\[
    f^{-1}(S)=\left\{x_{F_1}^{a_1}x_{F_2}^{a_2}\ldots x_{F_{\ell}}^{a_{\ell}}x_E^i : \substack{\emptyset=F_0\subsetneq F_1\subsetneq\cdots\subsetneq F_{\ell}\subsetneq E \text{ with }\rk(F_j)=s_j\\
    1\le a_j\le s_j-s_{j-1}-1  \text{ for }j=1,2\ldots,\ell\\
    0\le i\le r-s_{\ell}-1
    }\right\},
\]
Therefore,
\begin{align*}
    \phi_{S,r}(t)
    &=(t+t^2+\cdots+t^{s_1-1})\ldots(t+t^2+\cdots+t^{s_{\ell}-s_{\ell-1}-1})(1+t+\cdots+t^{r-s_{\ell}-1})\\
    &=t^{|S|}[s_1-1]_t[s_2-s_1-1]_t\ldots[s_{\ell}-s_{\ell-1}-1]_t[r-s_{\ell}]_t.
\end{align*}
In particular, when $S$ is an empty set, the inverse image $f^{-1}(\emptyset)=\{x_E^i~:~0\le i\le r-1\}$ and $\phi_{\emptyset,r}(t)=[r]_t$.

For (ii), the argument is analogous. The basis $\widetilde{\FY}(M)$ consists of monomials indexed by chains as above, with exponents satisfying
$1\le a_1\le \rk(F_1)$ and $1\le a_j\le \rk(F_j)-\rk(F_{j-1})-1$ for $j\ge 2$. 
Thus, we write
\[
\Hilb_G(\widetilde{A}(M)_{\mathbb{C}},t)
=\sum_{S\subseteq [r-1]} \psi_{S,r}(t)\,[\alpha_{\L(M)}(S)].
\]
Define $g:\widetilde{\FY}(M)\longrightarrow 2^{[r-1]}$ by
\[
g(x_{F_1}^{a_1}\cdots x_{F_\ell}^{a_\ell}x_E^i)
=\{\rk(F_1),\ldots,\rk(F_\ell)\}.
\]
Then counting $g^{-1}(S)$ for $S=\{s_1<\ldots<s_{\ell}\}\subseteq [r-1]$ gives 
\[
    \psi_{S,r}(t)=\begin{cases}
        t^{|S|}[s_1]_t[s_2-s_1-1]_t\cdots[s_{\ell}-s_{\ell-1}-1]_t[r-s_{\ell}]_t & \text{ if }S\neq \emptyset\\
        [r+1]_t & \text{ if } S=\emptyset
    \end{cases}.,
\]
as claimed.
\end{proof}

The following theorem establishes Conjecture \ref{conj: Real-rootedChow} and provides explicit expressions for the (equivariant) $\gamma$-coefficients of $A(M)$. It also generalizes the positivity of the (equivariant) Charney--Davis quantities for $A(M)$ obtained in \cite[Lemma~4.9]{Liao2023Two}.

\begin{thm} \label{thm:GammaChow}
For any loopless matroid $M$ of rank $r$ on the ground set $E$ with the lattice of flats $\L(M)$, the Chow ring $A(M)$ is equivariant $\gamma$-positive with the following $\gamma$-expansion
\begin{align*}
    \Hilb_{G}(A(M)_{\mathbb{C}},t) 
    &=\sum_{S\in\Stab([2,r-1])}[\beta_{\L(M)}(S)] t^{|S|}(1+t)^{r-1-2|S|}\\
    &=\sum_{k=0}^{\floor{\frac{r-1}{2}}}\left(\sum_{\substack{S\in\Stab([2,r-1])\\ |S|=k}}[\beta_{\L(M)}(S)]\right)t^k(1+t)^{r-1-2k},
\end{align*}
where $\beta_{\L(M)}(S)\cong \tilde{H}_{|S|-1}(\L(M)_S)$ as a $\mathbb{C}G$-module
\end{thm}

\begin{proof}
Notice that the polynomial $\phi_{S,r}(t)$ in Proposition \ref{prop:HilbChow} (i) is equal to $0$ if $s_j-s_{j-1}=1$ for some $1\le j\le \ell$.
Hence $\phi_{S,r}(t)$ takes nonzero value only when $S\in\Stab([2,r-1])$. 
The identity in Proposition \ref{prop:HilbChow} can be rephrased as follows:
\begin{align*}
    \Hilb_G(A(M)_{\mathbb{C}},t)
    &=\sum_{S\in \Stab([2,r-1])}\phi_{S,r}(t)[\alpha_{\L(M)}(S)]\\
    &=\sum_{S\in \Stab([2,r-1])}\phi_{S,r}(t)\sum_{T:T\subseteq S}[\beta_{\L(M)}(T)]\\
    &=\sum_{T\in\Stab([2,r-1])}\left(\sum_{T\subseteq S\subseteq [r-1]}\phi_{S,r}(t)\right)[\beta_{\L(M)}(T)],
\end{align*}
where the second equality follows from equation (\ref{eq: flagFtoH}). 
It turns out that this complicated sum has a neat reformulation:
\[
    \sum_{T\subseteq S\subseteq [r-1]}\phi_{S,r}(t)=t^{|T|}(1+t)^{r-1-2|T|}.
\]
This fact and its proof will be presented in Lemma \ref{lem:ChowIdentity}. Combining it with Theorem \ref{thm:BetaGenRep}, which states that $\beta_{\L(M)}(T)$ is isomorphic to $\tilde{H}_{|T|-1}(\L(M)_T)$ as $\mathbb{C}G$-modules, we obtain
\begin{align*}
    \Hilb_G(A(M)_{\mathbb{C}},t)
    &=\sum_{T\in\Stab([2,r-1])}[\beta_{\L(M)}(T)]t^{|T|}(1+t)^{r-1-2|T|}\\
    &=\sum_{T\in\Stab([2,r-1])}[\tilde{H}_{|T|-1}(\L(M)_T)]t^{|T|}(1+t)^{r-1-2|T|}.
\end{align*}
\end{proof}

\begin{rem}
The $\gamma$-expansion in Theorem \ref{thm:GammaChow}, together with its augmented counterpart in Theorem \ref{thm:GammaAug}, recovers and extends several known $\gamma$-expansions from \cite{ShareshianWachs2010}, \cite{ShareshianWachs2020}, and \cite{AthanasiadisGamma}.
Thus, the present approach yields a unified framework for these results; see Section \ref{sec:Uniform} for further details.
\end{rem}

\begin{rem}
It is relatively rare for equivariant $\gamma$-positivity to hold in such generality. To the best of the author's knowledge, another instance of this phenomenon appears in the context of \emph{Rees products of posets} by Athanasiadis  \cite{Athanasiadis2020EqGamma}.
In that setting, a key step in establishing equivariant $\gamma$-positivity is \cite[Lemma 3.1]{Athanasiadis2020EqGamma}, which involves a pair of identities closely resembling Lemma \ref{lem:ChowIdentity} and its augmented counterpart, Lemma \ref{lem:AugIdentity}.
It would be of interest to understand whether there is a common underlying structure that explains this parallel between the two settings.
\end{rem}

\begin{lem} \label{lem:ChowIdentity}
For a positive integer $n\ge 2$ and a subset $T\in\Stab([2,n-1])$, 
\begin{equation} \label{eq:ChowIdentity}
    \sum_{T\subseteq S\subseteq [n-1]}\phi_{S,n}(t)=t^{|T|}(1+t)^{n-1-2|T|},
\end{equation}
where $\phi_{S,n}(t)=0$ if $S\notin\Stab([2,n-1])$ and for $S=\{s_1<\cdots<s_\ell\}\in\Stab([2,n-1])$,
\[
    \phi_{S,n}(t)=t^{|S|}[s_1-1]_t[s_2-s_1-1]_t\ldots[s_\ell-s_{\ell-1}-1]_t[n-s_\ell]_t.
\]
\end{lem}

We give a combinatorial proof of Lemma \ref{lem:ChowIdentity}. 
This is done by double-counting a combinatorial model constructed below. 
The model is motivated by Angarone, Nathanson, and Reiner's combinatorial proof of \cite[Theorem 1.1]{ANR2023PermRepchow}.

\begin{defi} \label{def:ChowModel}
For $n\ge 2$, let $\mathcal{W}_n$ be the set of  sequences $\mathbf{w}=w_1 w_2\ldots w_{n-1}$ over the alphabet set $\{\bullet, \times, \underline{\phantom{~}} (\text{a blank space})\}$ satsifying the following conditions:
\begin{itemize}
    \item[(W1)] If $w_i=\bullet$, then $w_{i-1}=\times$ (hence $w_1\neq\bullet$ and there is no consecutive $\bullet$ in $\mathbf{w}$).
\end{itemize}
Let $S=\{s_1<s_2<\ldots<s_\ell\}\in \Stab([2,n-1])$ be the set of indices $s$ with $w_s=\bullet$.
\begin{itemize}
    \item[(W2)] If there is $s_i<j<s_{i+1}$ such that $w_j=\times$, then $w_j=w_{j+1}=\ldots=w_{s_{i+1}-1}=\times$.
    \item[(W3)] If there is $s_{\ell}<j<n$ such that $w_j=\times$, then $w_j=w_{j+1}=\ldots=w_{n-1}=\times$.
\end{itemize}
For $\mathbf{w}\in\mathcal{W}_n$, let $\Dot(\mathbf{w})=\{i\in [n-1]: w_i=\bullet\}$ and $\cro(\mathbf{w})$ be the number of crossings $\times$ in $\mathbf{w}$.
\end{defi}

\begin{proof}[Proof of Lemma \ref{lem:ChowIdentity}]
We prove (\ref{eq:ChowIdentity}) by double-counting
\[
\sum_{\substack{\mathbf{w}\in\mathcal{W}_n\\ \Dot(\mathbf{w})\supseteq T}} t^{\cro(\mathbf{w})}.
\]

\noindent \textbf{First counting.} Fix a subset $S\in\Stab([2,n-1])$, say $S=\{s_1<\ldots<s_\ell\}$.
For a sequence $\mathbf{w}=w_1\ldots w_{n-1}\in \mathcal{W}_n$ with $\Dot(\mathbf{w})=S$, we have $w_{s_{j}}=\bullet$ for all $j$. 
By (W1), each $\bullet$ forces $w_{s_j-1}=\times$. 
Between consecutive positions $s_{j-1}$ and $s_j$, the symbol at $s_j-1$ is fixed to be $\times$, and we may choose any number of additional consecutive $\times$ immediately to its left, ranging from $0$ to $s_j-s_{j-1}-2$. This contributes a factor
\[
t\left(1 + t + \cdots + t^{s_j-s_{j-1}-2}\right) = t[s_j-s_{j-1}-1]_t.
\]
A similar argument applies to the initial segment before $s_1$ and the final segment after $s_\ell$, with the latter contributing $[n-s_\ell]_t$.
It follows that, if $S$ is nonempty, 
\[
    \sum_{\substack{\mathbf{w}\in\mathcal{W}_n\\ \Dot(\mathbf{w})=S}}t^{\cro(\mathbf{w})}
    =
    t^{|S|}[s_1-1]_t[s_2-s_1-1]_t\cdots[s_\ell-s_{\ell-1}-1]_t[n-s_\ell]_t =\phi_{S,n}(t);
\]
While $S=\emptyset$, $\sum_{\substack{\mathbf{w}\in\mathcal{W}_n\\ \Dot(\mathbf{w})=\emptyset}}t^{\cro(\mathbf{w})}=1+t+\cdots+t^{n-1}=[n]_t=\phi_{\emptyset,n}(t)$.
Summing over all $S\supseteq T$ gives 
\[
    \sum_{\substack{\mathbf{w}\in\mathcal{W}_n\\ \Dot(\mathbf{w})\supseteq T}}t^{\cro(\mathbf{w})}=\sum_{T\subseteq S\subseteq [n-1]}\phi_{S,n}(t).
\]

\noindent \textbf{Second counting.} Fix $\mathbf{w}\in\mathcal{W}_n$ with $\Dot(\mathbf{w})\supseteq T$. For each $i\in T$, we have $w_i=\bullet$, and hence by (W1), $w_{i-1}=\times$. These forced $\times$ contribute a factor $t^{|T|}$.
The remaining positions, namely those not of the form $i$ or $i-1$ for $i\in T$, form a set of size $n-1-2|T|$. On these positions, we may independently choose whether to place a $\times$ or leave a blank, contributing a factor $(1+t)^{n-1-2|T|}$.
Given such a choice, the conditions (W2) and (W3) uniquely determine the placement of the remaining $\bullet$ symbols so that $\mathbf{w}\in\mathcal{W}_n$. Therefore,
\[
\sum_{\substack{\mathbf{w}\in\mathcal{W}_n\\ \Dot(\mathbf{w})\supseteq T}} t^{\cro(\mathbf{w})}
= t^{|T|}(1+t)^{n-1-2|T|}.
\]
Comparing the two expressions completes the proof.
\end{proof}

In the following, we illustrate the procedure in the proof of Lemma \ref{lem:ChowIdentity} by some examples.
\begin{exa}
Let $n=12$ and $T=\{3,9\}\in\Stab([2,11])$.

\medskip
\noindent
\textbf{First counting.}
Fix $S\supseteq T$.

\smallskip
\emph{Case 1:}
If $S=\{3,5,9\}$, we place $\bullet$ at positions $3,5,9$, and by (W1) we must place $\times$ at $2,4,8$. The remaining positions are initially blank:
\begin{center}
\begin{tabular}{c c c c c c c c c c c}
1 & 2 & 3 & 4 & 5 & 6 & 7 & 8 & 9 & 10 & 11\\
$\underline{\phantom{~}}$ & $\times$ & $\bullet$ & $\times$ & $\bullet$ & $\underline{\phantom{~}}$ & $\underline{\phantom{~}}$ & $\times$ & $\bullet$ & $\underline{\phantom{~}}$ & $\underline{\phantom{~}}$
\end{tabular}
\end{center}

Between consecutive $\bullet$'s, we may extend the forced $\times$'s to the left. For instance, between $w_5=\bullet$ and $w_9=\bullet$, we may choose
\[
(\underline{\phantom{~}},\underline{\phantom{~}},\times),\quad
(\underline{\phantom{~}},\times,\times),\quad
(\times,\times,\times),
\]
contributing $t+t^2+t^3=t[3]_t$. Proceeding similarly for all segments, we obtain
\[
\sum_{\substack{\mathbf{w}\in\mathcal{W}_{12}\\ \Dot(\mathbf{w})=\{3,5,9\}}}
t^{\cro(\mathbf{w})}
= t[2]_t \cdot t[1]_t \cdot t[3]_t \cdot [3]_t
= t^3[2]_t[1]_t[3]_t[3]_t.
\]

\smallskip
\emph{Case 2:}
If $S=\{3,9,11\}$, we place $\bullet$ at $3,9,11$ and $\times$ at $2,8,10$:
\begin{center}
\begin{tabular}{c c c c c c c c c c c}
1 & 2 & 3 & 4 & 5 & 6 & 7 & 8 & 9 & 10 & 11\\
$\underline{\phantom{~}}$ & $\times$ & $\bullet$ & $\underline{\phantom{~}}$ & $\underline{\phantom{~}}$ & $\underline{\phantom{~}}$ & $\underline{\phantom{~}}$ & $\times$ & $\bullet$ & $\times$ & $\bullet$
\end{tabular}
\end{center}

Since $11=n-1\in S$, there is no trailing segment after the last $\bullet$. The contributions from the segments give
\[
\sum_{\substack{\mathbf{w}\in\mathcal{W}_{12}\\ \Dot(\mathbf{w})=\{3,9,11\}}}
t^{\cro(\mathbf{w})}
= t[2]_t \cdot t[5]_t \cdot t[1]_t
= t^3[2]_t[5]_t[1]_t.
\]

\medskip
\noindent
\textbf{Second counting.}
We count sequences $\mathbf{w}$ with $\Dot(\mathbf{w})\supseteq T=\{3,9\}$. Then $w_3=w_9=\bullet$ and $w_2=w_8=\times$ are forced. The remaining positions
\[
\{1,4,5,6,7,10,11\}
\]
can be chosen independently to be either $\times$ or blank, contributing
\[
t^{|\{2,8\}|}(1+t)^{|\{1,4,5,6,7,10,11\}|}
= t^2(1+t)^{12-1-2\cdot 2}.
\]
For example, if we choose $w_1=w_4=w_5=w_{10}=\times$, we obtain
\begin{center}
\begin{tabular}{c c c c c c c c c c c}
1 & 2 & 3 & 4 & 5 & 6 & 7 & 8 & 9 & 10 & 11\\
$\times$ & $\times$ & $\bullet$ & $\times$ & $\times$ & $\underline{\phantom{~}}$ & $\underline{\phantom{~}}$ & $\times$ & $\bullet$ & $\times$ & $\underline{\phantom{~}}$
\end{tabular}.
\end{center}
There is then a unique way to insert additional $\bullet$'s so that $\mathbf{w}\in\mathcal{W}_{12}$, namely:
\begin{center}
\begin{tabular}{c c c c c c c c c c c}
1 & 2 & 3 & 4 & 5 & 6 & 7 & 8 & 9 & 10 & 11\\
$\times$ & $\times$ & $\bullet$ & $\times$ & $\times$ & $\bullet$ & $\underline{\phantom{~}}$ & $\times$ & $\bullet$ & $\times$ & $\bullet$
\end{tabular}.
\end{center}
This illustrates the identity
\[
\sum_{\substack{\mathbf{w}\in\mathcal{W}_{12}\\ \Dot(\mathbf{w})\supseteq T}}
t^{\cro(\mathbf{w})}
= t^2(1+t)^7.
\]
\end{exa}
 
The following theorem is the augmented counterpart of Theorem \ref{thm:GammaChow}. 
It generalizes the positivity of the (equivariant) Charney--Davis quantities for the augmented Chow ring $\widetilde{A}(M)$ obtained in \cite[Lemma~4.12]{Liao2023Two}.
The proof follows the same strategy as in the Chow case, with a modified combinatorial model.

\begin{thm} \label{thm:GammaAug}
For any loopless matroid $M$ of rank $r$ on the ground set $E$ with the lattice of flats $\L(M)$, the augmented Chow ring $\widetilde{A}(M)_{\mathbb{C}}$ is equivariant $\gamma$-positive with the following $\gamma$-expansion
\begin{align*}
    \Hilb_{G}(\widetilde{A}(M)_{\mathbb{C}},t) 
    &=\sum_{S\in\Stab([r-1])}[\beta_{\L(M)}(S)] t^{|S|}(1+t)^{r-2|S|}\\
    &=\sum_{k=0}^{\floor{\frac{r}{2}}}\left(\sum_{\substack{S\in\Stab([r-1])\\ |S|=k}}[\beta_{\L(M)}(S)]\right)t^k(1+t)^{r-2k}.
\end{align*}
where $\beta_{\L(M)}(S)\cong \tilde{H}_{|S|-1}(\L(M)_S)$ as $\mathbb{C}G$-modules.
\end{thm}

\begin{proof}
The argument is parallel to that of Theorem \ref{thm:GammaChow}.
By Proposition \ref{prop:HilbChow} (ii), we have
\[
\Hilb_G(\widetilde{A}(M)_{\mathbb{C}},t)
=\sum_{S\subseteq [r-1]} \psi_{S,r}(t)\,[\alpha_{\L(M)}(S)].
\]
Since $\psi_{S,r}(t)=0$ whenever $S$ contains consecutive integers, the sum reduces to $S\in\Stab([r-1])$. Thus, 
\begin{align*}
    \Hilb_G(\widetilde{A}(M)_{\mathbb{C}},t)
    &=\sum_{S\in\Stab([r-1])}\psi_{S,r}(t)[\alpha_{\L(M)}(S)]\\
    &=\sum_{S\in\Stab([r-1])}\psi_{S,r}(t)\sum_{T\subseteq S}[\beta_{\L(M)}(T)]\\
    &=\sum_{T\in\Stab([r-1])}\left(\sum_{T\subseteq S\subseteq [r-1]}\psi_{S,r}(t)\right)[\beta_{\L(M)}(T)].
\end{align*}
By Lemma \ref{lem:AugIdentity} below, the inner sum satisfies
\[
    \sum_{T\subseteq S\subseteq [r-1]}\psi_{S,r}(t)=t^{|T|}(1+t)^{r-2|T|},
\] 
Combining this with Theorem \ref{thm:BetaGenRep}, we obtain
\begin{align*}
    \Hilb_G(\widetilde{A}(M)_{\mathbb{C}},t)
    &=\sum_{T\in\Stab([r-1])}[\beta_{\L(M)}(T)]t^{|T|}(1+t)^{r-2|T|}\\
    &=\sum_{T\in\Stab([r-1])}[\tilde{H}_{|T|-1}(\L(M)_T)]t^{|T|}(1+t)^{r-2|T|}.
\end{align*}
\end{proof}

\begin{lem} \label{lem:AugIdentity}
For $n\ge 2$ and subset $T\in\Stab([n-1])$, 
\begin{equation} \label{eq:AugIdentity}
    \sum_{T\subseteq S\subseteq [n-1]}\psi_{S,n}(t)=t^{|T|}(1+t)^{n-2|T|}
\end{equation}
where $\psi_{S,n}(t)=0$ if $S\notin\Stab([n-1])$ and for $S=\{s_1<\cdots<s_{\ell}\}\in\Stab([n-1])$,
\[
    \psi_{S,n}(t)=\begin{cases}
        t^{|S|}[s_1]_t[s_2-s_1-1]_t\ldots[s_{\ell}-s_{\ell-1}-1]_t[n-s_{\ell}]_t  & \text{ if }S\neq\emptyset\\
        [n+1]_t  & \text{ if }S=\emptyset
    \end{cases}.
\]
\end{lem}

We use another combinatorial model slightly different from $\mathcal{W}_n$ in the non-augmented case to prove Lemma \ref{lem:AugIdentity}. 

\begin{defi}
For $n\ge 2$, let $\widetilde{\mathcal{W}}_n$ be the set of sequences $\mathbf{w}=w_0w_1w_2\ldots w_{n-1}$ of length $n$ over the alphabet set $\{\bullet, \times, \underline{\phantom{~}} (\text{a blank space})\}$ satisfying conditions
(W1), (W2), and (W3) from Definition \ref{def:ChowModel}.
\end{defi}
Condition (W1) implies that $w_0\neq \bullet$ and no two consecutive entries of $\mathbf{w}$ are both $\bullet$. 
As in the non-augmented case, we define $\Dot(\mathbf{w})$ to be the set of positions occupied by $\bullet$, and $\cro(\mathbf{w})$ to be the number of $\times$’s in $\mathbf{w}$ for $\mathbf{w}\in\widetilde{\mathcal{W}}_n$.

\begin{proof}[Proof of Lemma \ref{lem:AugIdentity}]
The argument is parallel to that of Lemma \ref{lem:ChowIdentity}. We compute
\[
\sum_{\substack{\mathbf{w}\in\widetilde{\mathcal{W}}_n\\ \Dot(\mathbf{w})\supseteq T}} t^{\cro(\mathbf{w})}
\]
in two ways.

\noindent \textbf{First counting.}
Fix $S=\{s_1<\cdots<s_\ell\}\in\Stab([n-1])$ with $S\supseteq T$. As in the non-augmented case, the structure of $\mathbf{w}\in\widetilde{\mathcal{W}}_n$ with $\Dot(\mathbf{w})=S$ yields
\[
\sum_{\substack{\mathbf{w}\in\widetilde{\mathcal{W}}_n\\ \Dot(\mathbf{w})=S}} t^{\cro(\mathbf{w})}
=\psi_{S,n}(t),
\]
where the contribution from each interval $[s_{j-1},s_j]$ is given by $[s_j-s_{j-1}-1]_t$ except when $j=1$ it is $[s_1]_t$, and the final segment contributes either $[n-s_\ell]_t$ or is absorbed into the special case $S=\emptyset$. Summing over all $S\supseteq T$ gives
\[
\sum_{\substack{\mathbf{w}\in\widetilde{\mathcal{W}}_n\\ \Dot(\mathbf{w})\supseteq T}} t^{\cro(\mathbf{w})}
=\sum_{T\subseteq S\subseteq [n-1]} \psi_{S,n}(t).
\]

\smallskip

\noindent \textbf{Second counting.}
We instead construct sequences with $\Dot(\mathbf{w})\supseteq T$ directly. First, for each $i\in T$, we force the pattern $w_i=\bullet$ and $w_{i-1}=\times$, contributing a factor $t^{|T|}$.
After removing these forced positions, there remain $n-2|T|$ free positions. 
In each such position, we independently choose either a $\times$ or a blank space, contributing a factor $(1+t)^{n-2|T|}$.
Given any such choice, there is a unique way to insert additional $\times$'s so that the resulting sequence lies in $\widetilde{\mathcal{W}}_n$. Hence,
\[
\sum_{\substack{\mathbf{w}\in\widetilde{\mathcal{W}}_n\\ \Dot(\mathbf{w})\supseteq T}} t^{\cro(\mathbf{w})}
= t^{|T|}(1+t)^{n-2|T|}.
\]
Equating the two expressions completes the proof.
\end{proof}

\begin{rem} \label{rem: ChowPoset}
In a forthcoming paper, Ferroni, Matherne, and Vecchi \cite{FMVChowPoly} generalize the Chow and augmented Chow polynomials of matroids (geometric lattice) to a more general class of posets. 
They prove these polynomials are $\gamma$-positive for Cohen--Macaulay posets, using a strategy analogous to that of the present paper, though without the combinatorial proofs of Lemma \ref{lem:ChowIdentity} and Lemma \ref{lem:AugIdentity} (See Section \ref{sec: Developments}).
\end{rem}

\section{Uniform and $q$-uniform matroids} \label{sec:Uniform}
In this section, we consider the special case where $M$ is the uniform matroid $U_{r,n}$ of rank $r$ and $G=\mathfrak{S}_n$ is the symmetric group. In this setting, the lattice of flats can be identified with a rank-selected subposet of the Boolean lattice $B_n=(2^{[n]},\subseteq)$.
The equivariant $\gamma$-coefficients $\tilde{H}_{|R|-1}((B_n)_R)$ carry a well-studied $\mathbb{C}\mathfrak{S}_n$-module structure for any $R\subseteq [n-1]$.
As a consequence, our results recover and extend several known $\gamma$-positivity phenomena in symmetric functions and permutation enumeration.

For background on representations of $\mathfrak{S}_n$, standard Young tableaux, and symmetric functions, we refer to Wachs \cite{WachsTool2007} and Stanley \cite{StanleyEC2}.
For permutation statistics such as $\exc$, $\inv$, and $\maj$, see Chapter 1 of Stanley \cite{StanleyEC1}.

Let \defin{$\ch$} denote the \defin{Frobenius characteristic map}, which assigns to each isomorphism class of $\mathbb{C}\mathfrak{S}_n$-modules a symmetric function of degree $n$. Let \defin{$\ps$} denote the \defin{stable principal specialization}, defined by substituting $x_i = q^{i-1}$ in a symmetric function.
In place of the equivariant Hilbert series, it is often convenient to use the graded Frobenius series to encode a graded $\mathfrak{S}_n$-representation as a polynomial with coefficients in symmetric functions. For a graded $\mathbb{C}\mathfrak{S}_n$-module $V=\bigoplus_i V_i$, the \defin{graded Frobenius series} of V is defined by
\[
\grFrob(V,t)\coloneqq \sum_i \ch(V_i)\, t^i.
\]

Let $\lambda$, $\mu$ be two partitions such that $\mu\subseteq\lambda$ (i.e. $\mu_i\le \lambda_i$ for all $i$). The connected skew shape $\lambda/\mu$ is said to be a \defin{ribbon} (or a \emph{skew hook} or a \emph{border strip}) if two consecutive rows always overlap in exactly one cell. A ribbon $H_{R,n}$ can be described in terms of the number $n$ of its cells and its \emph{descent set} $R\subseteq [n-1]$. Given a ribbon of $n$ cells, we label its cells with numbers $1$ through $n$, starting from the southwestern end cell, through the adjacent cell, ending at the northeastern end cell. We say a cell $i$ is a \defin{descent} of the ribbon if cell $i+1$ is directly above cell $i$. Then the collection of all its descents is called its \defin{descent set}.

\begin{exa}
Let $n=7$. Then the diagrams of $H_{\{2,4\},7}$ and $H_{\{1,3\},7}$ are 
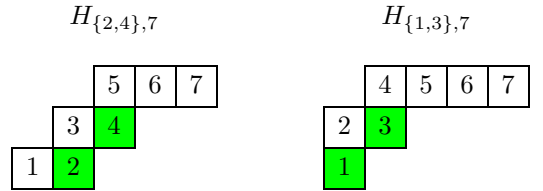
\[
\begin{array}{ccccc}
    H_{\{2,4\},7} &   &    &    & H_{\{1,3\},7} \\
                & &      &  &            \\
    \begin{ytableau}
    \none & \none & 5 & 6 & 7\\
    \none & 3 & *(green)4\\ 
    1     & *(green)2 
    \end{ytableau}  
    &   &     &     &
    \begin{ytableau}
    \none & 4 & 5 & 6 & 7\\
    2     & *(green)3 \\
    *(green) 1
    \end{ytableau}
\end{array}.
\]
\end{exa}
Let $S^{\lambda/\mu}$ be the \emph{Specht module} of skew shape $\lambda/\mu$. If $\lambda/\mu$ is a ribbon $H_{R,n}$, the corresponding Specht module $S^{H_{R,n}}$ is known as a \emph{Foulkes representation}. We summarize some facts of this representation from  \cite[Theorem 4.3]{Stanley1982GroupPoset} and \cite[Theorem 3.4.4]{WachsTool2007} in Theorem \ref{thm:RankHomologyBn}
(Note that Theorem \ref{thm:RankHomologyBn} (i) was first proved by Solomon \cite{Solomon1968} in a different expression for any finite Coxeter group).

For a partition $\lambda$ of $n$, denoted by $\lambda\vdash n$, let $SYT(\lambda)$ be the set of standard Young tableaux of shape $\lambda$. For $P\in SYT(\lambda)$, let $\DES(P)$ be the set of entries (called descents) $i\in[n-1]$ such that $i+1$ appears in a lower row than $i$ in $P$, and $\des(P)=|\DES(P)|$.

\begin{thm} \label{thm:RankHomologyBn}
Let $R\subseteq [n-1]$. 
\begin{itemize}
    \item[(i)] Then
\[
    \tilde{H}_{|R|-1}({(B_n)}_R)\cong_{\S_n} S^{H_{R,n}}\cong_{\S_n}\bigoplus_{\lambda\vdash n}b_{\lambda}(R)S^{\lambda},
\]    
where $b_{\lambda}(R)$ is the number of $P\in SYT(\lambda)$ for which $\DES(P)=R$.
    \item[(ii)] Let $\mathbb{P}$ denote the set of positive integers and $\mathbb{P}_n$ the set of words of length $n$ over $\mathbb{P}$. For $w=w_1\ldots w_n\in\mathbb{P}_n$, we say $i\in\DES(w)$ if $w_i>w_{i+1}$ and write $\x_w=x_{w_1}\ldots x_{w_n}$. Then the Frobenius characteristic is the ribbon Schur function
\[
    \ch(\tilde{H}_{|R|-1}({(B_n)}_R))=s_{H_{R,n}}=\sum_{\substack{w\in\mathbb{P}_n\\ \DES(w)=R}}\mathbf{x}_w,
\]

\end{itemize}

\end{thm}

Applying Theorem \ref{thm:GammaChow} to the uniform matroids $U_{r,n}$ together with Theorem \ref{thm:BetaGenRep}, one obtains the graded Frobenius series of $A(U_{r,n})_{\mathbb{C}}$. Applying the operator  $\prod_{i=1}^n(1-q^i)\ps(-)$ to the graded Frobenius series gives the Hilbert series of $q$-Uniform matroid $U_{r,n}(q)$ (See Section 3 in \cite{Liao2023Two} for the explanation). Note that the stable principal specialization on the ribbon Schur function gives
\[
    \prod_{i=1}^n(1-q^i)\ps(s_{H_{R,n}})=\sum_{\substack{\sigma\in\S_n \\ \DES(\sigma)=R}}q^{\maj(\sigma^{-1})}=\sum_{\substack{\sigma\in\S_n \\ \DES(\sigma)=R}}q^{\inv(\sigma)}
\]
(see \cite[Thm. 4.4]{ShareshianWachs2020} and its proof for a detailed discussion of the above identity). We obtain the following corollary, which generalizes Corollary 4.10 and 4.11 of \cite{Liao2023Two}    

\begin{cor} \label{cor:ChowUniform}
For any positive integer $n$ and $1\le r\le n$,
\begin{align*}
    \grFrob(A(U_{r,n})_{\mathbb{C}},t)
    &=\sum_{R\in\Stab([2,r-1])}s_{H_{R,n}}t^{|R|}(1+t)^{r-1-2|R|}\\
    &=\sum_{k=0}^{\lfloor\frac{r-1}{2}\rfloor}\left(\sum_{\substack{R\in\Stab([2,r-1])\\ |R|=k}}s_{H_{R,n}}\right)t^k(1+t)^{r-1-2k}
\end{align*}
and 
\begin{align*}
    \Hilb(A(U_{r,n}(q)),t)
    =\sum_{k=0}^{\lfloor\frac{r-1}{2}\rfloor}\xi_{r,n,k}(q)t^k(1+t)^{r-1-2k}
\end{align*}
where 
\[
    \xi_{r,n,k}(q)=\sum_{\sigma}q^{\maj(\sigma^{-1})}=\sum_{\sigma}q^{\inv(\sigma)}
\]
and the sum runs through all $\sigma\in\S_n$ for which $\DES(\sigma)\in \Stab([2,r-1])$ has $k$ elements.
\end{cor}

Applying Theorem \ref{thm:GammaAug} and the stable principal specialization as in Corollary \ref{cor:ChowUniform}, we obtain the augmented counterpart of Corollary \ref{cor:ChowUniform}, which generalizes Corollary 4.13 and 4.14 of \cite{Liao2023Two}.

\begin{cor} \label{cor:AugUniform}
For any positive integer $n$ and $1\le r\le n$,
\begin{align*}
    \grFrob(\widetilde{A}(U_{r,n})_{\mathbb{C}},t)
    &=\sum_{R\in\Stab([r-1])}s_{H_{R,n}}t^{|R|}(1+t)^{r-2|R|}\\
    &=\sum_{k=0}^{\lfloor\frac{r}{2}\rfloor}\left(\sum_{\substack{R\in\Stab([r-1])\\ |R|=k}}s_{H_{R,n}}\right)t^k(1+t)^{r-2k}
\end{align*}
and 
\[
    \Hilb(\widetilde{A}(U_{r,n}(q)),t)=\sum_{k=0}^{\lfloor\frac{r}{2}\rfloor}\widetilde{\xi}_{r,n,k}(q)t^k(1+t)^{r-2k},
\]
where 
\[
    \widetilde{\xi}_{r,n,k}(q)=\sum_{\sigma}q^{\maj(\sigma^{-1})}=\sum_{\sigma}q^{\inv(\sigma)}
\]
and the sum runs through all $\sigma\in\S_n$ for which $\DES(\sigma)\in\Stab([r-1])$ has $k$ elements.
\end{cor}

Recall from Corollary 4.2 and 4.4 of \cite{Liao2023Two} that
\begin{align}
    &\grFrob(A(U_{n-1,n}),t)=t^{-1}Q_n^0(\x,t),     \label{eq:Derangements}\\
    &\grFrob(A(U_{n,n}),t)=\grFrob(\widetilde{A}(U_{n-1,n}),t)=Q_n(\x,t), \label{eq:Permutations}\\
    &\grFrob(\widetilde{A}(U_{n,n}),t)=\widetilde{Q}_n(\x,t).  \label{eq:DecoratedPerms} 
\end{align}
Here $Q_n^0$, $Q_n$, $\widetilde{Q}_n$ are the \defin{Eulerian quasisymmetric functions} introduced by Shareshian and Wachs \cite{ShareshianWachs2010, ShareshianWachs2020}.
They admit combinatorial interpretations as enumerators over the sets of derangements $D_n$ (permutations with no fixed points), permutations $\mathfrak{S}_n$, and \emph{decorated permutations} $\widetilde{\mathfrak{S}}_n$ (permutations with two types of fixed points), respectively, satisfying
\[
D_n \subset \mathfrak{S}_n \subset \widetilde{\mathfrak{S}}_n.
\]
More precisely,
\begin{equation} \label{eq:InterpDerEulerQuas}
    Q_n^0(\x,t)=\sum_{\sigma\in D_n}F_{\DEX(\sigma),n}(\x)t^{\exc(\sigma)}, \quad Q_n(\x,t)=\sum_{\sigma\in \S_n}F_{\DEX(\sigma),n}(\x)t^{\exc(\sigma)},
\end{equation}
and
\begin{equation} \label{eq:InterpBinomEulerQuas}
    \widetilde{Q}_n(\x,t)=\sum_{\sigma\in\widetilde{\S}_n}F_{\DEX(\sigma),n}(\x)t^{\exc(\sigma)+1}.
\end{equation}
Here $F_{S,n}(\x)$ for $S\subseteq[n-1]$ denotes Gessel's \defin{fundamental quasisymmetric function} \footnote{Note that our definition of the fundamental quasisymmetric function is the reverse of the usual one. The usual definition is summed over $i_1\le i_2\le\cdots\le i_n$ with $j\in S$ implying $i_j<i_{j+1}$.}
\[
    F_{S,n}(\x)\coloneqq \sum_{\substack{i_1\ge i_2\ge \cdots\ge i_n\\ j\in S\Rightarrow i_j>i_{j+1}}}x_{i_1}x_{i_2}\ldots x_{i_n}
    \quad \text{ with }F_{\phi,0}\coloneqq 1.
\] 
See \cite{Liao2023Two} for the definition of the descent-like statistic $\DEX$ on $\widetilde{\mathfrak{S}}_n$ and further details on these symmetric functions. 
Finally, by Lemma 6.20 of \cite{Liao2023One}, applying the operator $\prod_{i=1}^n(1-q^i)\ps(-)$ to $Q_n^0$, $Q_n$, and $\widetilde{Q}_n$ yields the $q$-analogs of Eulerian type polynomials:
\[
    d_n(q,t)=\sum_{\sigma\in D_n}q^{\maj(\sigma)-\exc(\sigma)}t^{\exc(\sigma)}, \quad A_n(q,t)=\sum_{\sigma\in \S_n}q^{\maj(\sigma)-\exc(\sigma)}t^{\exc(\sigma)},
\]
and 
\[
    \widetilde{A}_n(q,t)\coloneqq 1+\sum_{k=1}^n\qbin{n}{k}A_k(q,t) =\sum_{\sigma\in \widetilde{\S}_n}q^{\maj(\sigma)-\exc(\sigma)}t^{\exc(\sigma)+1}.
\]

Let $NDD_n$ be the set of words $w\in\mathbb{P}_n$ with no double descents, i.e. $\DES(w)\in\Stab([n-1])$. Combining (\ref{eq:Derangements}) and Corollary \ref{cor:ChowUniform}, we recover the following collection of results from \cite[Theorem 7.3, Equation (7.9)]{ShareshianWachs2010} and \cite[Theorem 6.1]{ShareshianWachs2020}.

\begin{cor} \label{cor:Derangements}
For positive integer $n\ge 2$, we have 
\begin{align*}
    Q_n^0(\x,t)
    &=\sum_{R\in\Stab([2,n-2])}s_{H_{R,n}}t^{|R|}(1+t)^{n-1-2|R|}\\
    &=\sum_{\substack{w\in NDD_n\\ w_1\le w_2,~w_{n-1}\le w_n}}\x_w t^{\des(w)+1}(1+t)^{n-2-2\des(w)}\\
    &=\sum_{k=0}^{\lfloor\frac{n-2}{2}\rfloor}\xi_{n,k}^0(\x) t^{k+1}(1+t)^{n-2-2k},
\end{align*}
where 
\[
    \xi_{n,k}^0(\x)=\sum_{\substack{R\in\Stab([2,n-2])\\ |R|=k}}s_{H_{R,n}}=\sum_{\substack{w\in NDD_n\\ w_1\le w_2,~w_{n-1}\le w_n \\ \des(w)=k}}\x_w.
\]
And
\begin{align*}
    d_n(q,t)=\sum_{k=0}^{\lfloor\frac{n-2}{2}\rfloor}\xi_{n,k}^0(q)t^{k+1}(1+t)^{n-2-2k},
\end{align*}
where
\[
    \xi_{n,k}^0(q)=\sum_{\sigma}q^{\maj(\sigma^{-1})}=\sum_{\sigma}q^{\inv(\sigma)}
\]
and both of the sums run through $\sigma\in D_n$ which has $k$ descents and no double descents, no first, no last descent.
\end{cor}

Combining (\ref{eq:Permutations}) with Corollary \ref{cor:ChowUniform} and Corollary \ref{cor:AugUniform}, we recover the following collection of results from \cite[Theorem 7.3, Equation (7.8)]{ShareshianWachs2010} and \cite[Equations (1.4) and (6.1)]{LinussonShareshianWachs2012}.

\begin{cor} \label{cor:Permutations}
For positive integer $n\ge 1$, we have the following expressions
\begin{align*}
    Q_n(\x,t)
    &=\sum_{R\in\Stab([2,n-1])}s_{H_{R,n}}t^{|R|}(1+t)^{n-1-2|R|}=\sum_{R\in\Stab([n-2])}s_{H_{R,n}}t^{|R|}(1+t)^{n-1-2|R|}\\
    &=\sum_{\substack{w\in NDD_n\\ w_1\le w_2}}\x_w t^{\des(w)}(1+t)^{n-1-2\des(w)}=\sum_{\substack{w\in NDD_n\\ w_{n-1}\le w_n}}\x_w t^{\des(w)}(1+t)^{n-1-2\des(w)}\\
    &=\sum_{k=0}^{\lfloor\frac{n-1}{2}\rfloor}\xi_{n,k}(\x) t^{k}(1+t)^{n-1-2k},
\end{align*}
where 
\begin{align*}
    \xi_{n,k}(\x)
    &=\sum_{\substack{R\in\Stab([2,n-1])\\ |R|=k}}s_{H_{R,n}}
    =\sum_{\substack{R\in\Stab([n-2])\\ |R|=k}}s_{H_{R,n}}\\
    &=\sum_{\substack{w\in NDD_n\\ w_{1}\le w_2\\ \des(w)=k}}\x_w
    =\sum_{\substack{w\in NDD_n\\ w_{n-1}\le w_n\\ \des(w)=k}}\x_w.
\end{align*}
And 
\[
    A_n(q,t)=\sum_{k=0}^{\lfloor\frac{n-1}{2}\rfloor}\xi_{n,k}(q)t^{k+1}(1+t)^{n-1-2k},
\]
where 
\[
    \xi_{n,k}(q)=\sum_{\sigma}q^{\maj(\sigma^{-1})}=\sum_{\sigma}q^{\inv(\sigma)}
\]
and both of the sums run through either:
\begin{itemize}
    \item $\sigma\in\S_n$ with $k$ descents and no double descents, no first descent.
    \item $\sigma\in\S_n$ with $k$ descents and no double descents, no last descent.
\end{itemize}
\end{cor}

\begin{rem}
In \cite[Theorem 2.1]{AthanasiadisGamma}, Athanasiadis summarizes different interpretations of the $\gamma$-coefficients for the Eulerian polynomial. In particular, the interpretations are given by counting permutations in $\S_n$ related to the descent set in $\Stab([2,n-1])$ or in $\Stab([n-2])$. This fact is lifted to the $q$-analog and the symmetric function analog in Corollary \ref{cor:Permutations}. The agreement of the two interpretations regarding $\Stab([2,n-1])$ and $\Stab([n-2])$ was shown algebraically by (\ref{eq:Permutations}).   
\end{rem}

Combining (\ref{eq:DecoratedPerms}) and Corollary \ref{cor:AugUniform}, we recover  gives the following collection of results of \cite[Theorems 3.4 and 4.5]{ShareshianWachs2020}

\begin{cor}[cf. Shareshian and Wachs \cite{ShareshianWachs2020}, Athanasiadis \cite{AthanasiadisGamma}] \label{cor:DecoratedPerms}
For positive integer $n\ge 1$, we have 
\begin{align*}
    \widetilde{Q}_n(\x,t)
    &=\sum_{R\in\Stab([n-1])}s_{H_{R,n}}t^{|R|}(1+t)^{n-2|R|}\\
    &=\sum_{w\in NDD_n}\x_w t^{\des(w)}(1+t)^{n-2\des(w)}\\
    &=\sum_{k=0}^{\lfloor\frac{n}{2}\rfloor}\widetilde{\xi}_{n,k}(\x) t^k(1+t)^{n-2k},
\end{align*} 
where 
\[
    \widetilde{\xi}_{n,k}(\x)=\sum_{\substack{R\in\Stab([n-1])\\ |R|=k}}s_{H_{R,n}}=\sum_{\substack{w\in NDD_n \\ \des(w)=k}}\x_w.
\]
And 
\[
    \widetilde{A}_n(q,t)=\sum_{k=0}^{\lfloor\frac{n}{2}\rfloor}\widetilde{\xi}_{n,k}(q)t^k(1+t)^{n-2k},
\]
where
\[
    \widetilde{\xi}_{n,k}(q)=\sum_{\sigma}q^{\maj(\sigma^{-1})}=\sum_{\sigma}q^{\inv(\sigma)}
\]
and both of the sums run through $\sigma\in\S_n$ with $k$ descents and no double descents.
\end{cor}

\section{A problem proposed by Athanasiadis} \label{sec:AthanasiadisProblem}
In this section, we apply some results from Section \ref{sec:Uniform} to address a problem proposed by Athanasiadis \cite[Problem 2.5]{AthanasiadisGamma}.

For $\sigma\in\S_n$, we set  
\[
    \des^*(\sigma)=\begin{cases}
        \des(\sigma) & \text{ if }\sigma_1=1\\
        \des(\sigma)-1  & \text{ if }\sigma_1>1
    \end{cases}.
\]
Recall the following $(p,q)$-analog of the $\gamma$-expansion of the Eulerian polynomial in \cite[Theorem 2.3]{AthanasiadisGamma}.
\begin{thm} \label{thm:pqGammaPerm}
For all positive integer $n\ge 1$, 
\[
    \sum_{\sigma\in\S_n} p^{\des^*(\sigma)}q^{\maj(\sigma)-\exc(\sigma)}t^{\exc(\sigma)}=\sum_{k=0}^{\lfloor\frac{n-1}{2}\rfloor}\xi_{n,k}(p,q)t^{k}(1+t)^{n-1-2k}
\]    
where 
\[
    \xi_{n,k}(p,q)=\sum_{\sigma}p^{\des(\sigma^{-1})}q^{\maj(\sigma^{-1})}
\]
summing over permutations $\sigma\in\S_n$ with $k$ descents and no double descents, no last descent.
\end{thm}

Recall a similar $q$-analog of the $\gamma$-expansion of the binomial Eulerian polynomial $\widetilde{A}_n(t)$ found by Shareshian and Wachs \cite{ShareshianWachs2020} is given in Corollary \ref{cor:DecoratedPerms}. Athanasiadis proposed the following problem.
\begin{prob}[{\cite[Problem 2.5]{AthanasiadisGamma}}]
Find a $p$-analog of the $q$-polynomial in Corollary \ref{cor:DecoratedPerms} that is similar to Theorem \ref{thm:pqGammaPerm}.   
\end{prob}

The idea of finding such a $(p,q)$-analog of $\widetilde{A}_n(t)$ is similar to the proof of Theorem \ref{thm:pqGammaPerm}. 
But the key step is the $\S_n$-equivariant enumerator of decorated permutations in  (\ref{eq:InterpBinomEulerQuas})\footnote{This interpretation is proved in Liao \cite[Theorem 6.21]{Liao2023One}. The original definition given by Shareshian and Wachs \cite{ShareshianWachs2020} is $\widetilde{Q}_n(\x,t)\coloneqq h_n+t\sum_{k=1}^n h_{n-k} Q_k(\x,t)$.}.

Before answering Athanasiadis' question, let us first generalize Corollary 2.41 in \cite{AthanasiadisGamma}, which states the $\gamma$-positivity for the coefficient of Schur-expansions of $Q_n^0(\x,t)$ and $Q_n(\x,t)$. Note that $Q_n$ and $Q_n^0$ are related to the graded Frobenius series of the Chow ring of $U_{n,n}$ and $U_{n-1,n}$ respectively. Here we generalize these results to all uniform matroids $U_{r,n}$ and their augmented counterparts. 

Let
\[
\grFrob(A(U_{r,n})_{\mathbb{C}},t)=\sum_{\lambda\vdash n} P_{\lambda}^r(t)s_{\lambda} \text{ and }\grFrob(\widetilde{A}(U_{r,n})_{\mathbb{C}},t)=\sum_{\lambda\vdash n}\widetilde{P_\lambda^r}(t)s_{\lambda}.
\]
\begin{cor}\label{cor:UniformSchurCoeff}
For $\lambda\vdash n$, we have
\[
    P_{\lambda}^r(t)=\sum_{\substack{P\in SYT(\lambda)\\ \DES(P)\in\Stab([2,r-1])}} t^{\des(P)}(1+t)^{r-1-2\des(P)}=\sum_{k=0}^{\lfloor\frac{r-1}{2}\rfloor}\xi_{r,\lambda,k}~ t^k(1+t)^{r-1-2k}
\]
and 
\[
    \widetilde{P}_{\lambda}^r(t)=\sum_{\substack{P\in SYT(\lambda)\\ \DES(P)\in\Stab([r-1])}} t^{\des(P)}(1+t)^{r-2\des(P)}=\sum_{k=0}^{\lfloor\frac{r}{2}\rfloor}\widetilde{\xi}_{r,\lambda,k}~ t^k(1+t)^{r-2k}
\]
where $\xi_{r,\lambda,k}$ (respectively, $\widetilde{\xi}_{r,\lambda,k}$) is the number of tableaux $P\in SYT(\lambda)$ for which $\DES(P)\in\Stab([2,r-1])$ (respectively, $\Stab([r-1])$) has $k$ elements. In particular, $P_{\lambda}^r(t)$ and $\widetilde{P}_{\lambda}^r(t)$ are palindromic and unimodal for all $\lambda$ and $r$.
\end{cor}
\begin{proof}
Consider the non-augmented case. Apply Corollary \ref{cor:ChowUniform} and use Theorem \ref{thm:BetaGenRep} (i) to expand $s_{H_{R,n}}$ into Schur functions of regular shapes. We have
\[
    \grFrob(A(U_{r,n})_{\mathbb{C}},t)=\sum_{\lambda\vdash n}\left(\sum_{R\in\Stab([2,r-1])}b_{\lambda}(R)t^{|R|}(1+t)^{n-2|R|}\right)s_{\lambda}(\x).
\]
Comparing the coefficient of $s_{\lambda}(\x)$, we obtain
\[
    P_{\lambda}^r(t)=\sum_{R\in\Stab([2,r-1])}b_{\lambda}(R)t^{|R|}(1+t)^{r-1-2|R|}
    =\sum_{\substack{P\in SYT(\lambda)\\ \DES(P)\in\Stab([2,r-1])}}t^{\des(P)}(1+t)^{r-1-2\des(P)}.
\]
The augmented can be worked out similarly. 
\end{proof}

\begin{rem}
For general $1\le r\le n$, the symmetric function polynomials $\grFrob(A(U_{r,n})_{\mathbb{C}},t)$ and $\grFrob(\widetilde{A}(U_{r,n})_{\mathbb{C}},t)$ each has two difference expressions as shown in \cite[Theorem 4.1 and Theorem 4.3]{Liao2023Two}. Unfortunately, we do not have nice combinatorial expressions as the cases $r=n-1,~n$ in  (\ref{eq:InterpDerEulerQuas}) and (\ref{eq:InterpBinomEulerQuas}).   
\end{rem}
In particular, when $r=n$ in the augmented case of Corollary \ref{cor:UniformSchurCoeff}, it states that the Schur-expansion of the binomial Eulerian quasisymmetric function $\widetilde{Q}_n(\x,t)$ has $\gamma$-positive coefficients. This will be used in the proof of Theorem \ref{thm:pqAnalog}.

We used stable principal specialization to obtain many interesting $q$-analogs in the previous sections. To obtain a $(p,q)$-analog of $\widetilde{A}_n(t)$, we need a different specialization. Given a formal power series $f(\x)=f(x_1,x_2,\ldots)$ and a positive integer $m\ge 1$, the \emph{principal specialization} $\psm(f)$ of $f$ is obtain from $f$ by setting 
\[
    x_i=\begin{cases}
        q^{i-1} & \text{ if }1\le i\le m\\
        0       & \text{ if }i\ge m+1
    \end{cases}.
\]
Then the principal specialization of $F_{S,n}$ (see \cite[Lemma 5.2]{Gessel1993cycle}) and $s_{\lambda}(\x)$ (see \cite[Proposition 7.19.12]{StanleyEC2}) are given by the formulas
\begin{equation} \label{eq:psmF}
    \sum_{m\ge 1}\psm(F_{S,n}(\x))p^{m-1}=\frac{p^{|S|}q^{\sum_{i\in S}i}}{(1-p)(1-pq)\ldots(1-pq^n)}
\end{equation}
and
\begin{equation} \label{eq:psmSchur}
    \sum_{m\ge 1}\psm(s_{\lambda}(\x))p^{m-1}=\frac{\sum_{P\in SYT(\lambda)}p^{\des(P)}q^{\maj(P)}}{(1-p)(1-pq)\ldots(1-pq^n)}.
\end{equation}

For any decorated permutation $\sigma\in\widetilde{\S}_n$, we set 
\[
    \widetilde{\des}(\sigma)=\begin{cases}
        \des(\sigma) & \text{ if }\sigma_1=0 \text{ or }1\\
        \des(\sigma)-1 & \text{ if }\sigma_1>1
    \end{cases}.
\]

Now we are ready to state our $(p,q)$-analog for $\widetilde{A}_n(t)$.
\begin{thm} \label{thm:pqAnalog}
For a positive integer $n\ge 1$,
\begin{align*}
    \sum_{\sigma\in\widetilde{\S}_n}p^{\widetilde{\des}(\sigma)}q^{\maj(\sigma)-\exc(\sigma)}t^{\exc(\sigma)+1}
    &=\sum_{\substack{\sigma\in\S_n \\ \DES(\sigma)\in\Stab([n-1])}}p^{\des(\sigma^{-1})}q^{\maj(\sigma^{-1})}t^{\des(\sigma)}(1+t)^{n-2\des(\sigma)}\\
    &=\sum_{k=0}^{\lfloor\frac{n}{2}\rfloor}\widetilde{\xi}_{n,k}(p,q) t^k(1+t)^{n-2k}
\end{align*}
where 
\[
    \widetilde{\xi}_{n,k}(p,q)=\sum_{\sigma}p^{\des(\sigma^{-1})}q^{\maj(\sigma^{-1})}
\]
and the sum runs through $\sigma\in\S_n$ with $k$ descents and no double descents.
\end{thm}

\begin{proof}
 First applying principal specialization on (\ref{eq:InterpBinomEulerQuas}) and by (\ref{eq:psmF}), we have
\begin{align*}
    &\sum_{m\ge 1}\psm(\widetilde{Q}_n(\x,t))p^{m-1}\\
    =&\sum_{\sigma\in\widetilde{\S}_n}\sum_{m\ge 1}\psm(F_{\DEX(\sigma),n}(\x))p^{m-1}t^{\exc(\sigma)+1}\\
    =&\frac{\sum_{\sigma\in\widetilde{\S}_n}p^{\widetilde{\des}(\sigma)}q^{\maj(\sigma)-\exc(\sigma)}t^{\exc(\sigma)+1}}{(1-p)(1-pq)\ldots(1-pq^n)} \quad  \text{(by Lemma 6.20 in \cite{Liao2023One})}.
\end{align*}
On the other hand, 
\begin{align*}
    &\sum_{m\ge 1}\psm(\widetilde{Q}_n(\x,t))p^{m-1}\\
    =&\sum_{\lambda\vdash n} \widetilde{P}_{\lambda}^{n}(t)\sum_{m\ge 1}\psm(s_{\lambda}(\x))p^{m-1}\\
    =&\frac{\sum_{\lambda\vdash n} \widetilde{P}_{\lambda}^{n}(t)\sum_{P\in SYT(\lambda)}p^{\des(P)}q^{\maj(P)}}{(1-p)(1-pq)\ldots(1-pq^n)} \quad \text{(by (\ref{eq:psmSchur}))}.
\end{align*}
Therefore, we have 
\begin{align*}
    &\sum_{\sigma\in\widetilde{\S}_n}p^{\widetilde{\des}(\sigma)}q^{\maj(\sigma)-\exc(\sigma)}t^{\exc(\sigma)+1}
    =\sum_{\lambda\vdash n}\sum_{P\in SYT(\lambda)} \widetilde{P}_{\lambda}^{n}(t)p^{\des(P)}q^{\maj(P)}\\
    =&\sum_{\lambda\vdash n}\sum_{P\in SYT(\lambda)}\sum_{\substack{Q\in SYT(\lambda)\\ \DES(Q)\in\Stab([n-1])}}p^{\des(P)}q^{\maj(P)} t^{\des(Q)}(1+t)^{n-2\des(Q)} \quad \text{(by Corollary \ref{cor:UniformSchurCoeff})}\\
    =& \sum_{\substack{\sigma\in\S_n\\ \DES(\sigma)\in\Stab([n-1])}}p^{\des(\sigma^{-1})}q^{\maj(\sigma^{-1})}t^{\des(\sigma)}(1+t)^{n-2\des(\sigma)}.
\end{align*}
The last equality is obtained by applying the RSK algorithm described in Section 7.11 of Stanley \cite[Lemma 7.23.1]{StanleyEC2}.
\end{proof}

\section{A remark on braid maroids} \label{sec: Braid}

The \emph{Braid matroids} $K_n$ are the matroid associated with the braid arrangement $\{x_i-x_j=0\}_{1\le i<j\le n}$, and equivalently, the graphical matroid of the complete graph of $n$ vertices. Its lattice of flats is the intersection lattice of the braid arrangement, and is given by the partition lattice $\Pi_n$.
In \cite{Ferroni2024hilbert}, Ferroni, Matherne, Stevens, and Vecchi proposed a problem of studying the Hilbert series of the Chow and augmented Chow ring of $K_n$. Applying Theorem \ref{thm:GammaChow} and Theorem \ref{thm:GammaAug}, we have
\[
    \Hilb_G(A(K_n),t)=\sum_{S\in \Stab([2,n-2])}[\tilde{H}_{|S|-1}(({\Pi_n})_S)] t^{|S|}(1+t)^{n-2-2|S|}
\]
and 
\[
    \Hilb_G(\widetilde{A}(K_n),t)=\sum_{S\in \Stab([n-2])}[\tilde{H}_{|S|-1}(({\Pi_n})_S)] t^{|S|}(1+t)^{n-1-2|S|}.
\]
These give an explicit connection between the (augmented) Chow ring of $K_n$ and the top reduced homology of the rank-selected partition lattice $(\Pi_n)_S$. The study 
of $\S_n$-representation $\tilde{H}_{|S|-1}(({\Pi_n})_S)$ is initiated from Stanley's classical paper \cite{Stanley1982GroupPoset}, and has been studied in depth since then. A nice survey for this is Sundaram \cite{SundaramPartitionSurvey}. However, the $\S_n$-module structures of $\tilde{H}_{|S|-1}(({\Pi_n})_S)$ seem to vary a lot, and there is no unified characterization for every subset $S$ so far. Regarding the dimension (rank) of $\tilde{H}_{|S|-1}(({\Pi_n})_S)$, using the theory of EL-labeling, there is a general formula for all subsets $S$, see Sundaram \cite[Prop. 2.18]{SundaramPartition1994} or Bj\"{o}rner \cite[p.276, Exercise 7.33]{BjornerMatroidChapter}.

\section{New developments} \label{sec: Developments}
Since the preliminary version of this paper was posted on \texttt{arXiv}, there have been many interesting developments. 

\subsection{Chow polynomials for posets}
Ferroni, Matherne, and Vecchi \cite{FMVChowPoly} generalized the notion of Chow polynomials from matroids to locally finite, weakly ranked posets, placing them within the framework of Kazhdan--Lusztig--Stanley theory for posets.

Given such a poset $P$ and a choice of a distinguished element $\kappa$ in the incidence algebra of $P$, called a \emph{$P$-kernel}, they define the corresponding \emph{$\kappa$-Chow polynomial of $P$}.
For suitable choices of $\kappa$, this construction recovers and extends Chow-type invariants arising in the study of Eulerian posets and the Kazhdan--Lusztig theory of Coxeter groups, and allows for a systematic investigation of their positivity phenomena.

In particular, when $\kappa$ is the characteristic polynomial of $P$ and $P$ is a lattice of flats of a matroid $M$, one recovers the Chow polynomial of $M$.
As noted in Remark~\ref{rem: ChowPoset}, using an approach similar to ours, they proved that if $P$ is Cohen--Macaulay, then both the characteristic Chow and augmented Chow polynomials are $\gamma$-positive \cite[Theorem 4.25]{FMVChowPoly}. 
We may extend Theorems \ref{thm:GammaChow} and \ref{thm:GammaAug} to Cohen-Macaulay posets as well.

\begin{defi} \label{def: EqChowDef}
Let $P$ be a finite graded bounded poset of rank $r$, and $G$ be a group of automorphisms of $P$. Define the \defin{($G$-)equivariant Chow polynomial} and \defin{($G$-)equivariant augmented Chow polynomial} of $P$
\footnote{Recursive definitions of $\operatorname{H}_P(t)$ and $^{\mathrm{aug}}\operatorname{H}_P(t)$ analogous to the non-equivariant construction in \cite{Ferroni2024hilbert} are given by Matherne and Nepel (personal communication).},
respectively, to be the following polynomials with coefficients in $R_{\mathbb{C}}(G)$:
\[
    \operatorname{H}_P(t)\coloneqq \sum_{S\subseteq [r-1]}\phi_{S,r}(t)[\alpha_{P}(S)]
\]
and 
\[
    ^{\mathrm{aug}}\operatorname{H}_P(t)\coloneqq \sum_{S\subseteq [r-1]}\psi_{S,r}(t)[\alpha_{P}(S)],
\]
where $\phi_{S,r}(t)$ and $\psi_{S,r}(t)$ are as in Proposition \ref{prop:HilbChow}.
\end{defi}
When $G$ is trivial (or by taking dimensions of $\alpha_P(S)$), following from \cite[Theorem 4.1]{FMVChowPoly} and its left augmented counterpart, the above polynomials specialize to the non-equivariant Chow and augmented Chow polynomials of $P$.
The same arguments as in Theorems~\ref{thm:GammaChow} and \ref{thm:GammaAug} then yield the following extension.

\begin{thm}
Let $P$ be a finite bounded graded poset of rank $r$. 
Then:
\begin{itemize}
    \item[(i)] the equivariant Chow polynomial of $P$ admits the $\gamma$-expansion
    \[
        \operatorname{H}_P(t)=\sum_{S\in\Stab([2,r-1])}[\beta_{P}(S)] t^{|S|}(1+t)^{r-1-2|S|}, \text{ and}
    \]
    \item[(ii)] the equivariant augmented Chow polynomial of P admits $\gamma$-expansion
    \[
    ^{\mathrm{aug}}\operatorname{H}_P(t)=\sum_{S\in\Stab([r-1])}[\beta_{P}(S)] t^{|S|}(1+t)^{r-2|S|}.
    \]
\end{itemize}
In particular, if $P$ is a Cohen-Macaulay poset, then $\operatorname{H}_P(t)$ and $^{\mathrm{aug}}\operatorname{H}_P(t)$ are equivariant $\gamma$-positive.
\end{thm}

Along these lines, Ferroni, Matherne, and Vecchi further extend Conjecture \ref{conj: Real-rootedChow} to this borader setting. 
\begin{conj}[Conjecture 4.26 in \cite{FMVChowPoly}]
Let $P$ be a Cohen--Macaulay poset. The characteristic Chow polynomial of $P$ is real-rooted.    
\end{conj}
Note that the real-rootedness of the characteristic (left) augmented Chow polynomial of a Cohen--Macaulay poset would follow from \cite[Corollary~4.6]{FMVChowPoly} together with the above conjecture.

\subsection{Uniform matroids $U_{r,n}$ and their generalizations}
The coefficients of the Chow and augmented Chow polynomials of the Boolean matroid $U_{n,n}$ are counted, respectively, by the numbers of loopless Schubert matroids and Schubert matroids of the corresponding ranks. This follows from the bijection described in \cite[proof of Theorem~7.7]{EHL2022stellahedral} between the Feichtner--Yuzvinsky basis $\widetilde{\operatorname{FY}}(U_{n,n})$ and Schubert matroids.
Hoster \cite{HosterUniform2026} extended this combinatorial interpretation to Chow and augmented Chow polynomials of uniform matroids $U_{r,n}$ for all $r$.

Br\"{a}nd\'{e}n and Vecchi \cite{branden2025chow} proved that the Chow polynomials of uniform matroids $U_{r,n}$ are real-rooted, and also provided a new proof of the real-rootedness of the augmented Chow polynomials of $U_{r,n}$, independent of the earlier argument in \cite[Theorem~1.10]{Ferroni2024hilbert}.

Hoster and Stump \cite{HosterStumpRealrootedness2025} used the non-equivariant analogs of Theorems~\ref{thm:GammaChow} and \ref{thm:GammaAug} to prove real-rootedness of Chow and augmented Chow polynomials for posets $\widehat{P}=P\cup\{\hat{1}\}$ where $P$ is a simplicial poset with positive $h$-vector. In particular, their argument yields yet another proof of real-rootedness for uniform matroids.

\subsection{Braid matroids $K_n$}
Kannan and K\"{u}hne \cite{KannanBraid2025} obtained a closed formula for the generating function of the $\mathfrak{S}_n$-equivariant Chow polynomial $\operatorname{H}_{K_n}(t)$ of the braid matroid $K_n$, expressed in terms of the $\mathfrak{S}_n$-equivariant compactly supported Poincar\'e polynomial of $\mathcal{M}_{0,n+1}$. They further showed that the Chow polynomial of $K_n$ agrees with the Hilbert series of the Chow ring of the moduli space of genus-zero relative stable maps to $\mathbb{P}^1$.

\section*{Acknowledgements}
The author thanks Victor Reiner for suggesting the problem of extending the equivariant Charney--Davis results in \cite{Liao2023Two}, which motivated this work, and Sheila Sundaram for helpful discussions on the $\mathfrak{S}_n$-representation structure of $\tilde{H}_{|S|-1}((\Pi_n)_S)$.

\bibliographystyle{alpha}
\bibliography{bibliography}

\end{document}